\documentclass[fontsize=10pt]{article}

\usepackage[square,numbers,nonamebreak]{natbib}
\usepackage{setspace}

\usepackage[english]{babel}

\usepackage{chngcntr}

\usepackage[hang]{footmisc}
 at 11truept
\font\ssc=pplrc9d at 11 truept
\newcommand\qedbox{$\rlap{$\sqcap$}\sqcup$}

\usepackage[hmarginratio=1:1, bottom=2cm, top=2.5cm]{geometry}

\usepackage{graphicx}
\usepackage[usenames,dvipsnames,svgnames,table]{xcolor}
\definecolor{Maroon}{cmyk}{0, 0.87, 0.68, 0.32}
\definecolor{RoyalBlue2}{cmyk}{80,100,0,0.1}

\newcommand\auths[1]{\large \textsc{\textcolor{Maroon}{#1}}\setstretch{1.2}}
\newcommand\titl[1]{\center \linespread{1.1}\color{RoyalBlue2}\Large\textbf{ #1}\color{black}\bigskip} 
\renewcommand\abstract[1]{
\begin{center}
{\textbf{Abstract}}
\end{center}
{
\linespread{1.1}\fontsize{9pt}{-10pt}\selectfont #1}}

\usepackage[utf8]{inputenc}
\usepackage[sc,osf]{mathpazo}
\usepackage[euler-digits]{eulervm}
\usepackage[T1]{fontenc}
\usepackage{mathtools}

\DeclareSymbolFont{operators}{\encodingdefault}{ppl}{m}{n}
\DeclareMathAlphabet{\mathbf}{\encodingdefault}{ppl}{bx}{n}
\DeclareMathAlphabet{\mathit}{\encodingdefault}{ppl}{m}{it}

\usepackage{amsfonts}
\usepackage{amsmath}
\usepackage{ragged2e}
\usepackage{titlesec}
\usepackage{amssymb}

\renewcommand{\thesection}{\arabic{section}}
 
\titleformat{\section}{\medskip\bigskip\normalfont\Large\bf}{\thesection}{0.5em}{}
\titleformat{\subsection}{\smallskip\bigskip\normalfont\large\bf}{\thesubsection}{0.5em}{}

\usepackage{amsthm}
\newtheoremstyle{dotless}{}{}{\itshape}{}{\bfseries}{}{1em}{}

\theoremstyle{dotless}
\newtheorem{theo}{Theorem}

\newtheorem{lem}[theo]{Lemma}
\newtheorem{cor}[theo]{Corollary}

\renewenvironment{proof}{\smallbreak\noindent {\sc Proof \;---\;}}{\hfill\qedbox}

\counterwithout{theo}{section}
\numberwithin{theo}{section}

\DeclareOldFontCommand{\rm}{\normalfont\rmfamily}{\mathrm}
\DeclareOldFontCommand{\sf}{\normalfont\sffamily}{\mathsf}
\DeclareOldFontCommand{\tt}{\normalfont\ttfamily}{\mathtt}
\DeclareOldFontCommand{\bf}{\normalfont\bfseries}{\mathbf}
\DeclareOldFontCommand{\it}{\normalfont\itshape}{\mathit}
\DeclareOldFontCommand{\sl}{\normalfont\slshape}{\@nomath\sl}
\DeclareOldFontCommand{\sc}{\normalfont\scshape}{\@nomath\sc}

\begin{document}

\titl{A Note on the Series' of Ordinal Numbers
\footnote{The author is supported by GNSAGA (INdAM) and is a member of the non-profit association ``Advances in Group Theory and Applications'' (www.advgrouptheory.com).}}

\auths{Marco Trombetti}

\thispagestyle{empty}
\justify\noindent
\setstretch{0.3}
\abstract{The aim of this short note is to provide a proof to a statement of Sierpi\'nski concerning the number of possible sums of a series (of type $\lambda<\aleph_1$) of arbitrary ordinal numbers.}

\setstretch{2.1}
\noindent
{\fontsize{10pt}{-10pt}\selectfont {\it Mathematics Subject Classification (2020)}: 03E10}\\[-0.8cm]

\noindent 
\fontsize{10pt}{-10pt}\selectfont  {\it Keywords}: ordinal number; series\\[-0.8cm]

\setstretch{1.1}
\fontsize{11pt}{12pt}\selectfont
\section{Introduction}

It has been proved by Sierpi\'nski in \cite{Sierpinski} that if you have an infinite series $\{\alpha_i\}_{i\in\omega}$ of type $\omega$ of ordinal numbers, then there are only finitely many possibilities for the sums $\sum_{i\in\omega}\alpha_{\varphi(i)}$ where $\varphi:\omega\rightarrow\omega$ is a bijection. Moreover, it is stated without proof in the paper (and to my knowledge it is not proved anywhere else) that if you replace~$\omega$ by an ordinal number $\lambda<\aleph_1$, then the number of such possibilities is at most countable (and it can actually be infinite). I do not know the proof of Sierpi\'nski, but he did not sketch it and he did not add the usual words ``comme on le vérifie sans peine'' he used for something which is not difficult to prove. Sierpi\'nski was used to prove results which were only stated by prominent mathematicians of his time (you can find many instances of this fact in his book \cite{CardOrd}), and I thought it would have been nice to prove this result on the occasion of the 140th anniversary of the birth of Sierpi\'nski. My proof is certainly inspired by the ideas in~\cite{Sierpinski}, and actually goes a bit further the result stated by Sierpi\'nski, generalizing his main result (see Corollary \ref{cor1}).

\section{The proof}

Let $\lambda$ be an infinite ordinal number and let $\mathcal{S}=\{\alpha_i\}_{i\in\lambda}$ be a sequence of ordinal numbers. We define $\psi(\mathcal{S})$ as the smallest ordinal number $\mu$ such that $\mu=\sum_{\gamma\leq i}\alpha_i$ for some ordinal $\gamma<\lambda$ (this is the ``remainder'' of the series). For the sake of conciseness, we need to give some definitions: $$T(\mathcal{S})=\left\{\sum_{i\in\lambda}\alpha_{\varphi(i)}\,\big|\, \varphi:\lambda\rightarrow\lambda\;\textnormal{is a bijection}\right\},$$ $$T'(\mathcal{S})=\left\{\sum_{i\in\lambda}\alpha_{\varphi(i)}\,\big|\, \varphi:\lambda\rightarrow\lambda\;\textnormal{is an injection}\right\},$$ $$T''(\mathcal{S})=\left\{\sum_{i\in\lambda}\alpha_{\varphi(i)}\,\big|\, \varphi:\lambda\rightarrow\lambda\;\textnormal{is any map}\right\},$$ $$M(\mathcal{S})=\left\{\psi\big(\{\alpha_{\varphi(i)}\}_{i\in\lambda}\big)\;|\,\varphi:\lambda\rightarrow\lambda\;\textnormal{is a bijection}\right\}$$ and $$M'(\mathcal{S})=\left\{\psi\big(\{\alpha_{\varphi(i)}\}_{i\in\lambda}\big)\;|\,\varphi:\lambda\rightarrow\lambda\;\textnormal{is an injection}\right\}.$$ The result proved by Sierpi\'nski is that $T(\mathcal{S})$ is finite whenever $\lambda=\omega$; the result which is only stated by him asserts that $T(\mathcal{S})$ is at most countable whenever $\lambda<\aleph_1$. In what follows, we actually prove that $T''(\mathcal{S})$ is at most countable when $\lambda<\aleph_1$. Of course, if $T''(\mathcal{S})\supseteq T'(\mathcal{S})\supseteq T(\mathcal{S})$. The key point seems to be proving the result for $\lambda=\omega2=\omega+\omega$.

\begin{lem}\label{lem1}
If $\mathcal{S}=\{\alpha_i\}_{i\in \omega2}$ is a collection of ordinal numbers, then the cardinality of $T(\mathcal{S})$  is at most countable.
\end{lem}
\begin{proof}
Let $\mathcal{S}_1=\{\alpha_i\}_{i\in\omega}$ and $\mathcal{S}_2=\{\alpha_i\}_{\omega\leq i<\omega2}$. Let $\mathcal{T}_1$ be the set of all $i\in\omega$ for which the set $\mathcal{U}_1(\alpha)=\{j\in\omega\;|\;\alpha_j\geq\alpha_i\}$  is finite. If $\mathcal{T}_1$ were infinite, then we could find a strictly descending chain of ordinal numbers, a contradiction. Thus $\mathcal{T}_1$ is finite. Similarly, we define $\mathcal{T}_2$ as the set of all ordinal numbers $\omega\leq i<\omega2$ for which $\mathcal{U}_1(\alpha)=\{j\;|\;\omega\leq j<\omega2\;\textnormal{and}\;\alpha_j\geq\alpha_i\}$ is finite, and we observe as before that $\mathcal{T}_2$ is finite.

Let $\mu_1=\psi(\mathcal{S}_1)$, $\mu_2=\psi(\mathcal{S}_2)$ and $\mu_3=\psi(\mathcal{S}_3)$, where $\mathcal{S}_3=\{\alpha_{\varphi(i)}\}_{i\in\omega2}$ and\linebreak $\varphi: \omega2\rightarrow\omega2$ is a bijection such that $\varphi\circ\varphi=\varphi$, $\varphi(2n)=\omega+2n$ and $\varphi(\omega\epsilon+2n+1)=\omega\epsilon+2n+1$ for any non-negative integer $n$ and $\epsilon\in\{0,1\}$.

Let $f:\omega2\rightarrow\omega2$ be a bijection, and put $\mathcal{S}_f=\{\alpha_{f(i)}\}_{i\in\omega2}$. We claim that $\mu:=\psi(\mathcal{S}_{f})$ is $\mu_1$, $\mu_2$ or $\mu_3$. Define $$A=\{i\;|\;f(i)<\omega,\; i\geq\omega\}\quad\textnormal{and}\quad B=\{i\;|\;f(i)\geq\omega,\;i\geq\omega\}.$$ Let $m_1$ be the maximum of the (finite) set $f(\mathcal{T}_1\cup\mathcal{T}_2)$ and let $m$ be a non-negative integer such that $\omega+m>m_1$; of course we have $\mu=\sum_{i\geq \omega+m}\alpha_{f(i)}$. Now, choose $i\geq\omega+m$ and assume $f(i)$ belongs to $A$. Then (by the choice of $m$) there is $j<\omega$ such that $a_k\geq a_{f(i)}$ for any $k\geq j$, and so $\alpha_{\varphi(\omega+2k)}=\alpha_{2k}\geq\alpha_{f(i)}$ for any $k\geq j$. Similarly, if $f(i)\in B$, then there is $j<\omega$ such that $\alpha_{\omega+k}\geq \alpha_{f(i)}$ for any $k\geq j$, so $\alpha_{\varphi(\omega+2k+1)}=\alpha_{\omega+2k+1}\geq\alpha_{f(i)}$. 

Thus, if for instance $A$ and $B$ are infinite, then it easily follows that $\mu\leq\mu_3$, and a symmetric argument yields that $\mu_3\leq\mu$; so in this case $\mu=\mu_3$. If $A$ is infinite and $B$ is finite, then $\mu=\mu_1$. If $A$ is finite and $B$ is infinite, then $\mu=\mu_2$. The claim is proved.

A similar argument shows that $\psi\big(\{\alpha_{f(i)}\}_{i\in\omega}\big)$ is $\mu_1$, $\mu_2$ or $\mu_3$.

\smallskip

What we have proved so far shows that every ordinal in $T(\mathcal{S})$ is of the form $\beta+\mu_i+\gamma+\mu_j$, where $i,j\in\{1,2,3\}$ and both $\beta$ and $\gamma$ are sums of finitely many elements of $\mathcal{S}$. It is therefore clear that the cardinality of $T(\mathcal{S})$ is at most countable.
\end{proof}

\begin{cor}\label{cor1}
Let $\mathcal{S}=\{\alpha_i\}_{i\in\omega}$ be a collection of ordinal numbers. Then the cardinality of the set $T'(\mathcal{S})$ is at most countable.
\end{cor}
\begin{proof}
Suppose by contradiction that the cardinality of $T'(\mathcal{S})$ is uncountable, so also the cardinality of $M'(\mathcal{S})$ is uncountable.  Let $A$ (respectively, $B$) be the set of all injections $\varphi:\omega\rightarrow\omega$ such that $\omega\setminus\varphi(\omega)$ is infinite (finite). Clearly, $M'(\mathcal{S})=U\cup V$, where $$U=\big\{\psi\big(\{\alpha_{\tau(i)}\}_{i\in\omega}\big)\;|\;\tau\in A\big\}\quad\textnormal{and}\quad V=\big\{\psi\big(\{\alpha_{\tau(i)}\}_{i\in\omega}\big)\;|\;\tau\in B\big\}.$$ Now, the cardinality of $T(\mathcal{S})$ is finite by \cite{Sierpinski} and so the cardinality of $V$ is countable. Thus we only need to focus on the cardinality of $U$.

Choose a bijection $f:\omega2\rightarrow\omega$, and put $\mathcal{T}=\{\alpha_{f(i)}\}_{i\in\omega2}$. Clearly, the cardinality of~$M(\mathcal{T})$ is the same as the cardinality of $U$. But Lemma \ref{lem1} shows that the cardinality of $M(\mathcal{T})$ is at most countable, so we get a contradiction.
\end{proof}

\begin{theo}\label{theo1}
Let $\lambda$ be an ordinal number such that $\omega\leq\lambda<\aleph_1$. If $\mathcal{S}=\{\alpha_i\}_{i\in \lambda}$ is a collection of ordinal numbers, then the cardinality of the set $T'(\mathcal{S})$ is at most countable.
\end{theo}
\begin{proof}
We use transfinite induction on $\lambda$. In case $\lambda=\omega$, this is just Co\-rol\-la\-ry~\ref{cor1}. Suppose the result holds for an ordinal number $\mu$ such that $\omega\leq\mu<\lambda$ and let's prove it for $\lambda$. There is a sequence $\{\lambda_i\}_{i\in\omega}$ of ordinal numbers (eventually $0$) strictly smaller than $\lambda$ such that $\lambda_0=0$ and $\lambda=\sum_{i\in\omega}\lambda_i$. For each $i\in\omega$, define $\mathcal{S}_i=\{\alpha_{j}\}_{\lambda_i\leq j<\lambda_{i+1}}$. By induction, every $T'(\mathcal{S}_i)$ is countable, so such is the disjoint union $\mathcal{U}$ of all of them; write $\mathcal{U}=\{\beta_i\}_{i\in\omega}$. Then $T'(\mathcal{U})$ is countable by Corollary \ref{cor1} and we have done since $T'(\mathcal{U})=T'(\mathcal{S})$.
\end{proof}

\begin{cor}\label{cor2}
Let $\lambda$ be an ordinal number such that $\omega\leq\lambda<\aleph_1$. If $\mathcal{S}=\{\alpha_i\}_{i\in \lambda}$ is a collection of ordinal numbers, then the cardinality of the set $T''(\mathcal{S})$ is at most countable.
\end{cor}
\begin{proof}
Consider a series $\mathcal{S}_1=\{\beta_i\}_{i\in\lambda}$ in which for every ordinal $\alpha\in\mathcal{S}$ there are infinitely many indexes $i$ such that $\beta_i=\alpha$. In this case, $T'(\mathcal{S}_1)\supseteq T''(\mathcal{S})$, so we may apply Theorem \ref{theo1} and prove the statement.
\end{proof}

%
%
%

\bigskip\bigskip\bigskip

\renewcommand{\bibsection}{\begin{flushright}\Large
{
REFERENCES}\\
\rule{8cm}{0.4pt}\\[0.8cm]
\end{flushright}}

\bigskip\bigskip


\begin{flushleft}
\rule{8cm}{0.4pt}\\
\end{flushleft}

{
\sloppy
\noindent
Marco Trombetti

\noindent 
Dipartimento di Matematica e Applicazioni ``Renato Caccioppoli''

\noindent
Università degli Studi di Napoli Federico II

\noindent
Complesso Universitario Monte S. Angelo

\noindent
Via Cintia, Napoli (Italy)

\noindent
e-mail: marco.trombetti@unina.it 

}


\end{document}